\documentclass[11pt,leqno]{amsart}
\usepackage{amssymb}
\usepackage{xypic}
\begin{document}
\theoremstyle{plain}
\newtheorem{thm}{Theorem}[section]
\newtheorem*{thm1}{Theorem 1}
\newtheorem*{thm2}{Theorem 2}
\newtheorem{lemma}[thm]{Lemma}
\newtheorem{lem}[thm]{Lemma}
\newtheorem{cor}[thm]{Corollary}
\newtheorem{prop}[thm]{Proposition}
\newtheorem{propose}[thm]{Proposition}
\newtheorem{variant}[thm]{Variant}
\theoremstyle{definition}
\newtheorem{notations}[thm]{Notations}
\newtheorem{rem}[thm]{Remark}
\newtheorem{rmk}[thm]{Remark}
\newtheorem{rmks}[thm]{Remarks}
\newtheorem{defn}[thm]{Definition}
\newtheorem{ex}[thm]{Example}
\newtheorem{claim}[thm]{Claim}
\newtheorem{ass}[thm]{Assumption}
\numberwithin{equation}{section}
\newcounter{elno}                
\def\points{\list
{\hss\llap{\upshape{(\roman{elno})}}}{\usecounter{elno}}} 
\let\endpoints=\endlist


\catcode`\@=11
%
%
\def\opn#1#2{\def#1{\mathop{\kern0pt\fam0#2}\nolimits}} 
\def\bold#1{{\bf #1}}%
\def\underrightarrow{\mathpalette\underrightarrow@}
\def\underrightarrow@#1#2{\vtop{\ialign{$##$\cr
 \hfil#1#2\hfil\cr\noalign{\nointerlineskip}%
 #1{-}\mkern-6mu\cleaders\hbox{$#1\mkern-2mu{-}\mkern-2mu$}\hfill
 \mkern-6mu{\to}\cr}}}
\let\underarrow\underrightarrow
\def\underleftarrow{\mathpalette\underleftarrow@}
\def\underleftarrow@#1#2{\vtop{\ialign{$##$\cr
 \hfil#1#2\hfil\cr\noalign{\nointerlineskip}#1{\leftarrow}\mkern-6mu
 \cleaders\hbox{$#1\mkern-2mu{-}\mkern-2mu$}\hfill
 \mkern-6mu{-}\cr}}}
%
%

%
\def\:{\colon}
\let\oldtilde=\tilde
\def\tilde#1{\mathchoice{\widetilde{#1}}{\widetilde{#1}}%
{\indextil{#1}}{\oldtilde{#1}}}
\def\indextil#1{\lower2pt\hbox{$\textstyle{\oldtilde{\raise2pt%
\hbox{$\scriptstyle{#1}$}}}$}}
\def\pnt{{\raise1.1pt\hbox{$\textstyle.$}}}
%

%
\let\amp@rs@nd@\relax
\newdimen\ex@\ex@.2326ex
\newdimen\bigaw@l
\newdimen\minaw@
\minaw@16.08739\ex@
\newdimen\minCDaw@
\minCDaw@2.5pc
\newif\ifCD@
\def\minCDarrowwidth#1{\minCDaw@#1}
\newenvironment{CD}{\@CD}{\@endCD}
\def\@CD{\def\A##1A##2A{\llap{$\vcenter{\hbox
 {$\scriptstyle##1$}}$}\Big\uparrow\rlap{$\vcenter{\hbox{%
$\scriptstyle##2$}}$}&&}%
\def\V##1V##2V{\llap{$\vcenter{\hbox
 {$\scriptstyle##1$}}$}\Big\downarrow\rlap{$\vcenter{\hbox{%
$\scriptstyle##2$}}$}&&}%
\def\={&\hskip.5em\mathrel
 {\vbox{\hrule width\minCDaw@\vskip3\ex@\hrule width
 \minCDaw@}}\hskip.5em&}%
\def\verteq{\Big\Vert&&}%
\def\noarr{&&}%
\def\vspace##1{\noalign{\vskip##1\relax}}\relax\let\amp@rs@nd@&\iffalse}\fi
 \CD@true\vcenter\bgroup\relax\let\\=\cr\iffalse}\fi\tabskip\z@skip\baselineskip20\ex@
 \lineskip3\ex@\lineskiplimit3\ex@\halign\bgroup
 &\hfill$\m@th##$\hfill\cr}
\def\@endCD{\cr\egroup\egroup}
%
\def\>#1>#2>{\amp@rs@nd@\setbox\z@\hbox{$\scriptstyle
 \;{#1}\;\;$}\setbox\@ne\hbox{$\scriptstyle\;{#2}\;\;$}\setbox\tw@
 \hbox{$#2$}\ifCD@
 \global\bigaw@\minCDaw@\else\global\bigaw@\minaw@\fi
 \ifdim\wd\z@>\bigaw@\global\bigaw@\wd\z@\fi
 \ifdim\wd\@ne>\bigaw@\global\bigaw@\wd\@ne\fi
 \ifCD@\hskip.5em\fi
 \ifdim\wd\tw@>\z@
 \mathrel{\mathop{\hbox to\bigaw@{\rightarrowfill}}\limits^{#1}_{#2}}\else
 \mathrel{\mathop{\hbox to\bigaw@{\rightarrowfill}}\limits^{#1}}\fi
 \ifCD@\hskip.5em\fi\amp@rs@nd@}
\def\<#1<#2<{\amp@rs@nd@\setbox\z@\hbox{$\scriptstyle
 \;\;{#1}\;$}\setbox\@ne\hbox{$\scriptstyle\;\;{#2}\;$}\setbox\tw@
 \hbox{$#2$}\ifCD@
 \global\bigaw@\minCDaw@\else\global\bigaw@\minaw@\fi
 \ifdim\wd\z@>\bigaw@\global\bigaw@\wd\z@\fi
 \ifdim\wd\@ne>\bigaw@\global\bigaw@\wd\@ne\fi
 \ifCD@\hskip.5em\fi
 \ifdim\wd\tw@>\z@
 \mathrel{\mathop{\hbox to\bigaw@{\leftarrowfill}}\limits^{#1}_{#2}}\else
 \mathrel{\mathop{\hbox to\bigaw@{\leftarrowfill}}\limits^{#1}}\fi
 \ifCD@\hskip.5em\fi\amp@rs@nd@}
%
%
\newenvironment{CDS}{\@CDS}{\@endCDS}
\def\@CDS{\def\A##1A##2A{\llap{$\vcenter{\hbox
 {$\scriptstyle##1$}}$}\Big\uparrow\rlap{$\vcenter{\hbox{%
$\scriptstyle##2$}}$}&}%
\def\V##1V##2V{\llap{$\vcenter{\hbox
 {$\scriptstyle##1$}}$}\Big\downarrow\rlap{$\vcenter{\hbox{%
$\scriptstyle##2$}}$}&}%
\def\={&\hskip.5em\mathrel
 {\vbox{\hrule width\minCDaw@\vskip3\ex@\hrule width
 \minCDaw@}}\hskip.5em&}
\def\verteq{\Big\Vert&}
\def\novarr{&}
\def\noharr{&&}
\def\SE##1E##2E{\slantedarrow(0,18)(4,-3){##1}{##2}&}
\def\SW##1W##2W{\slantedarrow(24,18)(-4,-3){##1}{##2}&}
\def\NE##1E##2E{\slantedarrow(0,0)(4,3){##1}{##2}&}
\def\NW##1W##2W{\slantedarrow(24,0)(-4,3){##1}{##2}&}
\def\slantedarrow(##1)(##2)##3##4{%
\thinlines\unitlength1pt\lower 6.5pt\hbox{\begin{picture}(24,18)%
\put(##1){\vector(##2){24}}%
\put(0,8){$\scriptstyle##3$}%
\put(20,8){$\scriptstyle##4$}%
\end{picture}}}
\def\vspace##1{\noalign{\vskip##1\relax}}\relax\let\amp@rs@nd@&\iffalse}\fi
 \CD@true\vcenter\bgroup\relax\let\\=\cr\iffalse}\fi\tabskip\z@skip\baselineskip20\ex@
 \lineskip3\ex@\lineskiplimit3\ex@\halign\bgroup
 &\hfill$\m@th##$\hfill\cr}
\def\@endCDS{\cr\egroup\egroup}
%
\newdimen\TriCDarrw@
\newif\ifTriV@
\newenvironment{TriCDV}{\@TriCDV}{\@endTriCD}
\newenvironment{TriCDA}{\@TriCDA}{\@endTriCD}
\def\@TriCDV{\TriV@true\def\TriCDpos@{6}\@TriCD}
\def\@TriCDA{\TriV@false\def\TriCDpos@{10}\@TriCD}
\def\@TriCD#1#2#3#4#5#6{%
\setbox0\hbox{$\ifTriV@#6\else#1\fi$}
\TriCDarrw@=\wd0 \advance\TriCDarrw@ 24pt
\advance\TriCDarrw@ -1em
\def\SE##1E##2E{\slantedarrow(0,18)(2,-3){##1}{##2}&}
\def\SW##1W##2W{\slantedarrow(12,18)(-2,-3){##1}{##2}&}
\def\NE##1E##2E{\slantedarrow(0,0)(2,3){##1}{##2}&}
\def\NW##1W##2W{\slantedarrow(12,0)(-2,3){##1}{##2}&}
\def\slantedarrow(##1)(##2)##3##4{\thinlines\unitlength1pt
\lower 6.5pt\hbox{\begin{picture}(12,18)%
\put(##1){\vector(##2){12}}%
\put(-4,\TriCDpos@){$\scriptstyle##3$}%
\put(12,\TriCDpos@){$\scriptstyle##4$}%
\end{picture}}}
\def\={\mathrel {\vbox{\hrule
   width\TriCDarrw@\vskip3\ex@\hrule width
   \TriCDarrw@}}}
\def\>##1>>{\setbox\z@\hbox{$\scriptstyle
 \;{##1}\;\;$}\global\bigaw@\TriCDarrw@
 \ifdim\wd\z@>\bigaw@\global\bigaw@\wd\z@\fi
 \hskip.5em
 \mathrel{\mathop{\hbox to \TriCDarrw@
{\rightarrowfill}}\limits^{##1}}
 \hskip.5em}
\def\<##1<<{\setbox\z@\hbox{$\scriptstyle
 \;{##1}\;\;$}\global\bigaw@\TriCDarrw@
 \ifdim\wd\z@>\bigaw@\global\bigaw@\wd\z@\fi
 \mathrel{\mathop{\hbox to\bigaw@{\leftarrowfill}}\limits^{##1}}
 }
 \CD@true\vcenter\bgroup\relax\let\\=\cr\iffalse}\fi
 \tabskip\z@skip\baselineskip20\ex@
 \lineskip3\ex@\lineskiplimit3\ex@
 \ifTriV@
 \halign\bgroup
 &\hfill$\m@th##$\hfill\cr
#1&\multispan3\hfill$#2$\hfill&#3\\
&#4&#5\\
&&#6\cr\egroup%
\else
 \halign\bgroup
 &\hfill$\m@th##$\hfill\cr
&&#1\\%
&#2&#3\\
#4&\multispan3\hfill$#5$\hfill&#6\cr\egroup
\fi}
\def\@endTriCD{\egroup} 
\newcommand{\mc}{\mathcal} 
\newcommand{\mb}{\mathbb} 
\newcommand{\surj}{\twoheadrightarrow} 
\newcommand{\inj}{\hookrightarrow} \newcommand{\zar}{{\rm zar}} 
\newcommand{\an}{{\rm an}} \newcommand{\red}{{\rm red}} 
\newcommand{\Rank}{{\rm rk}} \newcommand{\codim}{{\rm codim}} 
\newcommand{\rank}{{\rm rank}} \newcommand{\Ker}{{\rm Ker \ }} 
\newcommand{\Pic}{{\rm Pic}} \newcommand{\Div}{{\rm Div}} 
\newcommand{\Hom}{{\rm Hom}} \newcommand{\im}{{\rm im}} 
\newcommand{\Spec}{{\rm Spec \,}} \newcommand{\Sing}{{\rm Sing}} 
\newcommand{\sing}{{\rm sing}} \newcommand{\reg}{{\rm reg}} 
\newcommand{\Char}{{\rm char}} \newcommand{\Tr}{{\rm Tr}} 
\newcommand{\Gal}{{\rm Gal}} \newcommand{\Min}{{\rm Min \ }} 
\newcommand{\Max}{{\rm Max \ }} \newcommand{\Alb}{{\rm Alb}\,} 
\newcommand{\GL}{{\rm GL}\,} 
\newcommand{\ie}{{\it i.e.\/},\ } \newcommand{\niso}{\not\cong} 
\newcommand{\nin}{\not\in} 
\newcommand{\soplus}[1]{\stackrel{#1}{\oplus}} 
\newcommand{\by}[1]{\stackrel{#1}{\rightarrow}} 
\newcommand{\longby}[1]{\stackrel{#1}{\longrightarrow}} 
\newcommand{\vlongby}[1]{\stackrel{#1}{\mbox{\large{$\longrightarrow$}}}} 
\newcommand{\ldownarrow}{\mbox{\Large{\Large{$\downarrow$}}}} 
\newcommand{\lsearrow}{\mbox{\Large{$\searrow$}}} 
\renewcommand{\d}{\stackrel{\mbox{\scriptsize{$\bullet$}}}{}} 
\newcommand{\dlog}{{\rm dlog}\,} 
\newcommand{\longto}{\longrightarrow} 
\newcommand{\vlongto}{\mbox{{\Large{$\longto$}}}} 
\newcommand{\limdir}[1]{{\displaystyle{\mathop{\rm lim}_{\buildrel\longrightarrow\over{#1}}}}\,} 
\newcommand{\liminv}[1]{{\displaystyle{\mathop{\rm lim}_{\buildrel\longleftarrow\over{#1}}}}\,} 
\newcommand{\norm}[1]{\mbox{$\parallel{#1}\parallel$}} 
\newcommand{\boxtensor}{{\Box\kern-9.03pt\raise1.42pt\hbox{$\times$}}} 
\newcommand{\into}{\hookrightarrow} \newcommand{\image}{{\rm image}\,} 
\newcommand{\Lie}{{\rm Lie}\,} 
\newcommand{\CM}{\rm CM}
\newcommand{\sext}{\mbox{${\mathcal E}xt\,$}} 
\newcommand{\shom}{\mbox{${\mathcal H}om\,$}} 
\newcommand{\coker}{{\rm coker}\,} 
\newcommand{\sm}{{\rm sm}} 
\newcommand{\tensor}{\otimes} 
\renewcommand{\iff}{\mbox{ $\Longleftrightarrow$ }} 
\newcommand{\supp}{{\rm supp}\,} 
\newcommand{\ext}[1]{\stackrel{#1}{\wedge}} 
\newcommand{\onto}{\mbox{$\,\>>>\hspace{-.5cm}\to\hspace{.15cm}$}} 
\newcommand{\propsubset} {\mbox{$\textstyle{ 
\subseteq_{\kern-5pt\raise-1pt\hbox{\mbox{\tiny{$/$}}}}}$}} 
\newcommand{\sB}{{\mathcal B}} \newcommand{\sC}{{\mathcal C}} 
\newcommand{\sD}{{\mathcal D}} \newcommand{\sE}{{\mathcal E}} 
\newcommand{\sF}{{\mathcal F}} \newcommand{\sG}{{\mathcal G}} 
\newcommand{\sH}{{\mathcal H}} \newcommand{\sI}{{\mathcal I}} 
\newcommand{\sJ}{{\mathcal J}} \newcommand{\sK}{{\mathcal K}} 
\newcommand{\sL}{{\mathcal L}} \newcommand{\sM}{{\mathcal M}} 
\newcommand{\sN}{{\mathcal N}} \newcommand{\sO}{{\mathcal O}} 
\newcommand{\sP}{{\mathcal P}} \newcommand{\sQ}{{\mathcal Q}} 
\newcommand{\sR}{{\mathcal R}} \newcommand{\sS}{{\mathcal S}} 
\newcommand{\sT}{{\mathcal T}} \newcommand{\sU}{{\mathcal U}} 
\newcommand{\sV}{{\mathcal V}} \newcommand{\sW}{{\mathcal W}} 
\newcommand{\sX}{{\mathcal X}} \newcommand{\sY}{{\mathcal Y}} 
\newcommand{\sZ}{{\mathcal Z}} \newcommand{\ccL}{\sL} 
 \newcommand{\A}{{\mathbb A}} \newcommand{\B}{{\mathbb 
B}} \newcommand{\C}{{\mathbb C}} \newcommand{\D}{{\mathbb D}} 
\newcommand{\E}{{\mathbb E}} \newcommand{\F}{{\mathbb F}} 
\newcommand{\G}{{\mathbb G}} \newcommand{\HH}{{\mathbb H}} 
\newcommand{\I}{{\mathbb I}} \newcommand{\J}{{\mathbb J}} 
\newcommand{\M}{{\mathbb M}} \newcommand{\N}{{\mathbb N}} 
\renewcommand{\P}{{\mathbb P}} \newcommand{\Q}{{\mathbb Q}} 

\newcommand{\R}{{\mathbb R}} \newcommand{\T}{{\mathbb T}} 
\newcommand{\U}{{\mathbb U}} \newcommand{\V}{{\mathbb V}} 
\newcommand{\W}{{\mathbb W}} \newcommand{\X}{{\mathbb X}} 
\newcommand{\Y}{{\mathbb Y}} \newcommand{\Z}{{\mathbb Z}} 
\title[Nondiscreteness of $F$-thresholds] 
{Nondiscreteness of $F$-thresholds} 
\author{V. Trivedi} 
\address{School of Mathematics, Tata Institute of 
Fundamental Research, Homi Bhabha Road, Mumbai-400005, India} 
\email{vija@math.tifr.res.in} 
\date{}

\begin{abstract}{We give  examples of two dimensional normal $\Q$-Gorenstein
graded domains, where the set of $F$-thresholds of the maximal ideal is not discrete,
thus answering a question by Musta\c{t}\u{a}-Takagi-Watanabe. 

We also prove that, for a two dimensional standard graded domain $(R, {\bf m})$ over a 
field of characteristic $0$, with  graded ideal $I$, if $({\bf m}_p, I_p)$ is 
a {\em reduction mod} $p$ of $({\bf m}, I)$ then  
 $c^{I_p}({\bf m}_p) \neq c^I_{\infty}({\bf m})$ implies  
$c^{I_p}({\bf m}_p)$ has $p$ in the denominator.

}\end{abstract} 

\maketitle\section{Introduction}

Let $(R, {\bf m})$ be a Noetherian local ring of positive characteristic $p$. For an ideal 
$I$ of  $R$, a set of invariants of singularities 
in positive characteristic, called $F$-thresholds,
 were introduced by [MTW] as follows
$$ \{F\mbox{-thresholds of}~~I\} := 
\{c^J(I)\mid J \subseteq {\bf m}~~\mbox{such that}~~I\subseteq \mbox{Rad}(J) \},$$
where $c^J(I) := \lim_{e\to \infty}{\max\{r\mid I^r\nsubseteq J^{[p^e]}\}}/{p^e}.$
In [MTW], it was shown that 
for regular local $F$-finite rings, 
the $F$-thresholds of an ideal 
 coincide with the $F$-jumping numbers of the generalized test 
ideals of $I$ (introduced by [HY]), which  are analogous to the jumping numbers of 
a multiplier ideal in characteristic $0$.
The first $F$-jumping number (introduced by [TaW]
under the name  $F$-pure threshold), denoted by $\mbox{fpt}$, corresponds to the first 
jumping number of  the associated multiplier ideal  and is called log 
canonical thershold of $I$.
The set of the jumping numbers, for a given ideal, 
is known to be discrete and rational.

Here we consider the following question 
by Musta\c{t}\u{a}-Takagi-Watanabe (Question~2.11 in [MTW]).

\vspace{5pt}
{\bf Question}.\quad Given an ideal $(0) \neq I \subseteq {\bf m}$, could 
there exist finite accumulation points for the set of $F$-thresholds of $I$? 

\vspace{5pt}

In the case of regular rings  (with some additional mild conditions), the set of $F$-jumping 
numbers for $I$ is equal to the set of $F$-thresholds 
$\{c^J(I)\}_J$ of $I$ (Corollary~2.3 in [BMS2]). On the other hand, in such cases,
it has been proven that 
the $F$-jumping numbers are 
discrete and rational (see [BMS1], [BMS2],  [KLZ]) (in fact, as pointed out in [BMS2], 
the discreteness of  
the set of $F$-jumping
numbers implies the rationality statement due to the 
fact that if $\lambda $ is an
$F$-jumping number, then so are the fractional parts
of $p^e\lambda $, for all $e \geq1$).

Though the discreteness of the set of $F$-jumping numbers are known in some 
singular cases too {\em e.g.} when the ring is $F$-finite  normal $\Q$-Gorenstein domain 
([GrS], [BSTZ], [KSSZ], [ST]), we cannot 
conclude the same for $F$-thresholds as they can be in general different from 
   the $F$-jumping numbers, as shown by Example~2.5 in [TaW], where 
the ring  $R = k[x,y,z]/(xy-z^2)$,
and the first $F$-jumping number of ${\bf m} = (x, y, z)$,
 $\mbox{fpt}_{\bf m}({\bf m}) <c^{\bf m}({\bf m})$, the first $F$-threshold of ${\bf m}$.

However
 when $R$ is a direct summand of a regular 
$F$-finite domain $S$, then the set $\{c^J(I)\}_J$ is known to be a  
discrete set of rational points (Proposition~4.17 in [HMNb]). 
Here the authors 
extend the theory  
of Bernstein-Sato polynomial to the direct summands of regular rings, while
for regular rings the authors in [MTW] relate the Bernstein-Sato polynomials to 
the $F$-jumping numbers and the $F$-thresholds. 
Now in [HMNb], each $c^J(I)$  is identified 
with $c^{JS}(IS)$ and hence is an 
$F$-jumping number of $IS$.

In particular,  in all of the above cases, the  $F$-thresholds of an ideal have
been studied by identifying them with the  $F$-jumping numbers of some ideal in a 
regular ring where such set is discrete and consist of rational numbers.

In [TrW], using the theory of the Hilbert-Kunz density functions 
for graded rings and Frobenius semistability properties  of vector bundles on
projective  curves, 
we had shown that 
in dimension two, the  $F$-thresholds of the  maximal ideal at graded ideals 
can be expressed in terms of the  
Harder-Narasimhan slopes of the associated syzygy bundles. As a result, 
we had deduced that  
the set
$\{c^{I}({\bf m})\mid I~~\mbox{is graded}\}_I$ consists of rational points.
 
In this paper, we use this new viewpoint to  show  that such a set 
can  have accumulation points. More precisely we prove the following 

\begin{thm}\label{t1}
Given a prime $p$ and an integer $g>1$, there is a two-dimensional 
standard graded normal $\Q$-Gorenstein domain 
$(R, {\bf m})$ (a cone over a nonsingular curve of genus $g$)
 over an
algebraically field of char $p>0$ and  a sequence of ${\bf m}$-primary graded 
ideals $\{I_m\}_{m\geq 0}$ such that, the $F$-threshold of ${\bf m}$ at $I_m$,
$$c^{I_m}({\bf m}) = \frac{3}{2}+\frac{(g-1)}{p^{m+m_0}d},
\quad\mbox{for}\quad m\geq 0,$$
where $d = e_0(R,{\bf m})$ and $m_0\geq 0$ is an integer such that $p^{m_0} <g$.
Moreover, each $I_m$ is generated by three elements, each  of degree $1$ in $R$.
\end{thm}

This answers the above question (of [MTW]) affirmatively.
In particular we have the following

\begin{cor}\label{c1}Given a prime   $p$ and an integer $g>1$,
there exists a  two dimensional standard  graded 
normal $\Q$-Gorenstein domain $R$ with the graded maximal ideal ${\bf m}$ 
such that the set of $F$-thresholds of ${\bf m}$ has 
accumulation points, where   
$\mbox{Proj}~R = X$ is a 
 nonsigular projective curve of genus $g$ over  a field of char~$p$.

Moreover there is a strictly decreasing sequence consisting of $F$-thresholds of 
${\bf m}$; thus, the $F$-thresholds of an ideal need not satisfy the descending chain condition.
\end{cor}

For the proof of Theorem~\ref{t1},  we crucially use the following construction by 
D. Gieseker in [G]. For a given $p$ and $g> 1$,
there exists a family $X$ of stable curves of genus $g$ over $\Spec k[[t]]$
 ($k$ is an 
algebraically closed field of $\Char~p$) with smooth generic fiber, and a closed fiber 
with particular singularities. By taking  a specific 
representation of $G$ (analogus to the representation arising from 
a Schottky uniformization for a compact Riemann surface of genus $g$), where 
$G$ is  the group of covering transformations of $Y_0$ (and
where $Y_0$ is the 
universal cover over the special fiber $X_0$ of $X$), Gieseker  constructed a 
rank $2$ vector bundle $F_1$ on the generic fiber $X_K$ ($K = k((t))$ with 
an explicit 
Harder-Narasimhan filtration. Moreover the bundle $F_1$,  associated to 
the  representation of $G$, comes equipped with a sequence $\{F_k\}_{k\geq 1}$ 
of bundles such that 
$F^*F_{k+1} = F_k$.

From this sequence  we construct  a set of  
vector bundles  with the 
similar properties such that the new set is also a  
{\em bounded family} of 
bundles on the curve $X_K$. 
 By choosing 
$\sL =$  the power of the canonical bundle of the curve, we ensure that the 
coordinate  ring (corresponding to the embedding of the curve by $\sL$) 
is $\Q$-Gorenstein.

\vspace{5pt}

Next, we consider some behaviour  of the $F$-thresholds of  
{\em reductions mod}~$p$, as $p$ varies, from our view point 
(relating $F$-thresholds  to 
vector bundles). We recall that Theorem~3.4 and Proposition~3.8 of [HY] imply that, 
for $R=\Z[X_1, \ldots, X_n]$ and $I\subseteq {\bf m} = (X_1, \ldots, X_n)$, we have 
a formula for the log canonical threshold in terms of $F$-pure threshols:
$$\mbox{lct}_{\bf m}(I) = \lim_{p\to \infty}
\mbox{fpt}_{{\bf m}_p}(I_p),$$
where ${\bf m}_p$ and $I_p$ are  {\em reductions mod}~$p$ of ${\bf m}$ and $I$, 
respectively.

 K.Schwede asked the following 
question. Assuming $\mbox{fpt}_{{\bf m}_p}(f_p)\neq \mbox{lct}_{\bf m}(f)$, is the 
denominator of $\mbox{fpt}_{{\bf m}_p}(f_p)$ (in its reduced form)
 a multiple of $p$?

In [CHSW] the authors explored the implication of the following two conditions:
(1) the characteristic does not divide the denominator of the $F$-pure threshold.
(2) The $F$-pure threshold and log canonical threshold coincide.
Theorem~A in [CHSW] and also 
the example~4.5 in [MTW] imply that for 
 an explicit (nonhomogeneous) polynomial $f$ in a polynomial ring
(note that here  the $F$-pure 
threshold $\mbox{fpt}_{{\bf m}_p}(f_p) = c^{{\bf m}_p}(f_p)$), the above two conditions 
could be distinct.

On the other hand, there are examples (see [CHSW] for the references) 
of homogeneous polynomials 
$f$ of specific types where the two conditions are equivalent. 
In [BS] Proposition~5.4, it was shown that for  a 
homogeneous 
polynomial $f$ of degree $d$ in $R = k[X_0, \ldots, X_n]$ 
(where $R/(f)$ is an isolated singularity), if $p\geq nd-d-n$ then either
$c^{{\bf m}_p}(f_p) = (n+1)/d$, or  the denominator of $c^{{\bf m}_p}(f_p)$ 
is a power of $p$. In other words
$$c^{{\bf m}_p}(f_p) \neq  \mbox{lct}_{\bf m}(f) \implies~~
\mbox{the denominator of}~~c^{{\bf m}_p}(f_p)~~\mbox{is a power of}~~p.$$

In this context, here we prove the following

\vspace{5pt}

\begin{thm}\label{t2} Let $R$ be a two dimensional standard graded domain over 
an algebraically 
closed field $k$ of char $0$ and let $I\subset R$ be a graded ideal of finite colength.
Let $R_p$, $I_p$ and ${\bf m}_p$ denote a reduction mod $p$ of $R$, $I$ 
and ${\bf m}$ respectively, where $\Char~R_p = p$.
Then, for $p\gg 0$, 
$$c^{I_p}({\bf m}_p) \neq c^I_{\infty}({\bf m}) \implies
c^{I_p}({\bf m}_p) =  c^I_{\infty}({\bf m}) +\frac{a}{p b},$$
for some   $a, b\in \Z_{+}$ such that $\mbox{g.c.d.}(a, p) = 1$. Moreover
$0< a/b \leq 4(g-1)(r-1)$, where $r+1 =$ the minimal generators of $I$ and 
$g =$ the genus of $\mbox{Proj}~R$.
Therefore 
$$\mbox{for}~~p\gg0,~~~c^{I_p}({\bf m}_p) \neq c^I_{\infty}({\bf m}) \implies
c^{I_p}({\bf m}_p) =   {a_1}/{pb_1},$$
 where  $a_1, b_1\in \Z_{+}$ and $\mbox{g.c.d.}(a_1, p) = 1$.
\end{thm}

However,  there exist examples (Remark~\ref{r3}) where
 the  denominators (in its reduced form) of $c^{{\bf m}_p}({\bf m}_p)$  
is divisible by $p$ but is  not a power of $p$.

We recall that  the existence of $c^I_{\infty}({\bf m}) := 
\lim_{p\to \infty}c^{I_p}({\bf m}_p)$  was shown in  
Theorem~5.5 of [TrW], and   $c^I_{\infty}({\bf m})$ and
$c^{I_p}({\bf m}_p)$  
were given, respectively,  in terms of the minimal  HN slope of a $\mu$-reduction bundle
(which is a char~$0$ invariant of the pair $(R, I)$)
 and the minimal strong HN slope of 
a strong $\mu$-reduction bundle for $(R, I)$.
For the proof of the above theorem we use the relation 
between these two bundles.

In Section~2 we recall the required basic theory of Harder-Narasimhan filtrations of 
vector bundles on curves, and also results from [TrW]. In Section~3 and Section~4 we
prove Theorem~\ref{t1} and Theorem~\ref{t2} respectively.

\section{preliminaries}
We recall few generalities about Harder-Narasimhan filtration of vector 
bundles on curves.

\vspace{5pt}

\noindent{\bf Definition/Notations}.\quad Let $X$ be a nonsingular curve over 
a field $k$.
For a vector bundle $V$ on  $X$, 
the degree of $V$ is $\deg~V = \deg(\wedge^{\rank V}V)$ and the slope of $V$ is 
$\mu(V) = \deg~V/\rank~V$.
 A vector bundle $V$ is {\em semistable} if for every  subbundle $W\subseteq V$, we have
$\mu(W)\leq \mu(V)$.

Every bundle has the unique {\em HN (Harder-Narasimhan) filtration}, which is 
a filtration 
\begin{equation}\label{e2}0=E_0\subset E_1\subset\cdots \subset E_n = V\end{equation} 
such that $\mu(E_1) > \mu(E_2/E_1) >\cdots >\mu(E_n/E_{n-1})$. We call  
$\mu(E_i/E_{i-1})$ an HN slope of of $V$.

If $\Char~k = p>0$ then the HN filtration  of $V$ is 
{\em strong} HN filtration if each $\mu(E_i/E_{i-1})$ 
is strongly semistable, {\em i.e.}, 
$F^{m*}(E_i/E_{i-1})$ is semistable for every $m^{th}$-iterated Frobenius map 
$F^{m}:X\longto X$.
For every vector bundle $V$ there exists $m\geq 0$ (Theorem~2.7 of [L]) such that 
$F^{m*}V$ has strong HN filtration.

For the vector bundle $V$ with the HN filtration~(\ref{e2}), we  
denote $\mu_{min}(V) = \mu(V/E_{n-1})$.

If $m$ is an integer such that 
$F^{m*}V$ achieves the  strong HN filtration then we denote
 $a_{min}(V) = \mu_{min}(F^{m*}(V)/p^{m})$.

\vspace{5pt}

We recall the following definitions and results from [TrW].
\begin{notations}\label{n1}Let $(R, I)$ be a standard graded pair defined over an
algebraically closed  
field, where $R$ is a
two dimensional domain and $I$ is generated by homogeneous elements of degress
$d_1, \ldots, d_s$. Let $X= \mbox{Proj}~S$, where $S$ is the normalization of
$R$ in its quotient field. Let
\begin{equation}\label{e3}0\longto V_0\longto M_0 = \oplus_{i=1}^s\sO_X(1-d_i)
\longto \sO_{X}(1)
\longto 0\end{equation}
 be the canonical sequence of $\sO_X$-modules.

If $M_0$ has  the HN filtration, 
$0 = M_{l_1} \subset M_{{l_1}-1} \subset  \cdots \subset M_0 = M$ then let
$0 = V_{l_1} \subseteq V_{l_1-1} \subseteq \cdots 
\subseteq V_1 \subseteq V_0$
denote the induced filtration on $V_0$, where $V_i = M_i\cap V_0$.
Note that this need not be the HN filtration of $V_0$.
 \end{notations}

\begin{defn}\label{mur}
\begin{enumerate}
\item The sequence~(\ref{e3}) has the {\em $\mu$-reduction
at $t$} if there exists
$0\leq t < l_1$ such that
\begin{enumerate}
\item for every $0\leq i < t$, the canonical sequence
$0\longto V_i\longto M_i \longto \sL\longto 0$ is exact and
 $\mu_{min}(V_i) = \mu_{min}(M_i)$, and
\item $\mu_{min}(V_t) <\mu_{min}(M_t)$.
\end{enumerate}
We call $V_t$ the {\em $\mu$-reduction bundle} for the sequence (\ref{e3}) and a 
$\mu$-reduction bundle for the pair $(R,I)$.
\item We say (provided $\Char~k = p>0$), the sequence  (\ref{e3}) has {\em the
strong $\mu$-reduction at $t_0$}, if for some choice of
$m_1> 0$ such that $F^{m_1*}(V_0)$ has the strong HN filtration,
the sequence
\begin{equation}\label{fe2}
0\longto F^{m_1*}V_0 \longto F^{m_1*}M_0 \longto
F^{m_1*}\sO_X(1)\longto 0,\end{equation}
has $\mu$-reduction sequence at $t_0$.
Note that $F^{m_1*}(V_i) = F^{m_1*}M_i\cap F^{m_1*}(V_0)$, for all 
$0\leq i\leq l_1$.
By Proposition~4.6 of [TrW],  the  sequence~(\ref{e3}) does have
the $\mu$-reduction for some $t <l_1$ and does have the 
strong $\mu$-reduction for some $t_0$. Moreover   
 $t_0\leq t$.
\end{enumerate}
\end{defn}

\begin{rmk}\label{r2}We recall Theorem~4.12, Remark~4.13~(1) and Lemma~5.2 from [TrW].
\begin{enumerate}\item 
If the sequence~(\ref{e3}) has the strong $\mu$-reduction at $t_0$ then 
the $F$-threshold of ${\bf m}$ at $I$ is 
$c^I({\bf m}) = 1- a_{min}(V_{t_0})/d.$

Moreover, if  $d_1 = \cdots = d_s$ (where $d_i$ as in Notation~\ref{n1}) then 
$c^I({\bf m}) = 1-a_{min}(V_{0})/d$.

\item If the sequence~(\ref{e3}) has $\mu$-reduction as $t$ 
(in charactersitic $0$) 
then,  for $p\gg 0$, where $({\bf m}_p, I_p)$ denote 
a {\em reduction mod} $p$ of $({\bf m}, I)$ and $V_t^p$ and $V_{t-1}^p$ denote a
{\em reduction mod}~$p$ of $V_t$ and $V_{t-1}$, respectively. 
$$\mbox{either}\quad~~c^{I_p}({\bf m}_p) = 1-{a_{min}(V_t^p)}/{d}\quad
\mbox{and}\quad~~c^I_{\infty}({\bf m}) =  1-{\mu_{min}(V_t^p)}/{d},$$
or $c^{I_p}({\bf m}_p) = 1-{a_{min}(V_{t-1}^p)}/{d}\quad\mbox{and}\quad
c^I_{\infty}({\bf m}) =  1-{\mu_{min}(V_{t-1}^p)}/{d}$.

In particular, for $p\gg 0$, we have $c^I_{\infty}({\bf m}) \leq 
c^{I_p}({\bf m}_p)$.

Though the above equalities hold,  the 
strong $\mu$-reduction bundle may not be a {\em reduction mod}~$p$ of the 
$\mu$-reduction bundle.
Also though $\mu$-reduction bundle $V_t$ 
may not occur in the HN filtration of $V_0$, the 
 $\mu_{min}(V_t)$ is equal to one of the HN slopes of
 $V_0$. Similarly, for the strong $\mu$-reduction bundle $V_{t_0}$, 
$\mu_{min}(F^{m_1*}V_{t_0})$ is equal to one of the HN slopes of $F^{m_1*}(V_0)$. 
\end{enumerate}
\end{rmk}

\section{Nondiscreteness of $F$-thresholds}

We recall a result by Gieseker [G].

\vspace{5pt}
 
{\bf Theorem~1} (Gieseker).\quad{\em For each prime  $p>0$ and integer $g>1$, 
there is  a nonsingular projective curve $X$ of genus $g$ over an
algebraically closed  field of char $p$ and  a semistable 
vector-bundle $V$ of degree $0$  such that $F^*V$ is not  semistable.}

\vspace{5pt}

Bundles of positive degree with such properties 
have been constructed by J.P.~Serre and H.~Tango. But for our result 
we use the other properties of this bundle, which were  proved by Gieseker. 
We recall the relevant results from  [G]:

\vspace{5pt}

For each $g>1$ and each algebraically closed field $k$ of char~$p$,  
 there
is a family of stable curves $X$ of genus $g$ over $k[[t]]$, such that the special fiber 
$X_0$ is 
a rational curve over $k$ with $g$ nodes and is $k$-split degenerate, and 
the  generic fiber $X_K$ is smooth and geometrically connected, 
where $K$ is the quotient field of $k[[t]]$.
Now if $Y_0$ is  the universal covering space of the special fiber $X_0$ and 
$G$ is the group of the covering transformations of $Y_0$
over $X_0$, then  (Proposition 2,~[G])
any representation $\rho$ of $G$ on $K^n$ gives a rank $n$  bundle 
$F_{\rho}$ on $X$ such that  the pull back bundle $F_1$ on $X_K$ comes 
 with a  sequence of bundles 
$F_1, F_2, F_3,\ldots $ such that $F^*F_{k+1} \simeq F_k$. 
Now, by making  a specific choice of a representation $\rho$ (attributed to
Mumford by [G]) of the group $G$ on $K^2$,
Gieseker derives (Lemma 4,~[G]) a rank $2$ bundle $F_\rho$ of degree $0$ on $X$ 
and 
an exact sequence 
$$0\longto L\longto F_{\rho} \longto L^{-1}\longto 0,$$
 where $\deg L = g-1$. Now pull back of $L$ to $X_K$ gives the HN filtration 
$L\subset F_1$ and also a sequence of bundles $F_1$, $F_2$, $F_3, \ldots$
such that $F^*F_{k+1} = F_k$.
By a simple argument it follows (Lemma 5,~[G]) that if $g\leq p^{k-1}$ then 
$F_k$ is semistable.
Hence one can choose a (unique) bundle $V$ from the set $\{F_k\}_{k\geq 1}$ 
such that $V$ is semistable and $F^*V$  is not semistable.

In the following lemma  we consider a modified version of such a family $\{F_m\}_m$ of 
bundles.

\begin{lemma}\label{l1}Given an integer $g>1$ and a prime $p$, there is 
  a nonsingular curve $X$ of genus $g$
over a  field of characteristc $p$ and a family of bundles $\{E_m\}_{m\geq 0}$
such that 
\begin{enumerate}
\item  $\rank~E_m =2$ and $\mbox{det}(E_m) = \sO_X$, for $m\geq 0$ and
\item for each $E_m$, the number $m\geq 0$ is the least integer such that the bundle
$F^{m*}E_m$ is not semistable. Moreover 
 the HN filration of $F^{m*}E_m$ is 
$$L_m\subset F^{m*}E_m,~~\mbox{where}~~\deg(L_m) = (g-1)/p^{m_0},$$ 
for some $m_0\geq 0$ where  
$p^{m_0} < g$.
\item There exists a very ample line bundle $\sL$ on $X$, such that 
for every $m\geq 0$, the bundle $E_m\tensor \sL$ is generated by 
its global sections.
\end{enumerate}
In particular $\{E_m\tensor\sL\}_{m\geq 0}$ is a bounded family.
\end{lemma}
\begin{proof} The results in [G] (see the above discussion) give the following:
for given $g>1$ and $p$, there is 
  a nonsingular curve $X$ of genus $g$
over an algebraically closed  field of $\Char~p$ and a family of bundles $\{F_m\}_{m\geq 1}$
such that 

\begin{enumerate}
\item  $F_m$ is of rank~$2$ and of $degree~0$, for $m\geq 1$ and 
\item $F^*F_{m+1} = F_m$, and $F_m$ semistable if $g\leq p^{m-1}$,
\item $F_1$ has the HN filtration 
$L\subset F_1$, where $\deg L = g-1$ and $\deg F_1 = 0$.
\end{enumerate}

Hence, for some $m_0 \geq 0$, there is a  (unique) 
bundle  $F_{m_0+2}\in \{F_k\}_{k\geq 1}$ which is 
semistable and $F^*F_{m_0+2} = 
F_{m_0+1}$ is not
semistable. Since $\Pic^0(X)$ (the set of degree $0$ line bundles on $X$) is an 
abelian variety, (Application~2, page~59 in [Mu1])
the map 
$$n_X:\Pic^0(X) \longto \Pic^0(X),~~\mbox{given by}~~\sL\mapsto \sL^{n}~~\mbox{is
surjective},$$  
where we denote 
$\sL^n = \sL^{\tensor n}$ (and $\sL^{-n} = (\sL^{-1})^{\tensor n}$).  
Therefore, for each $m$,  we can choose $\sL_m\in \Pic^0(X)$ such 
that 
$\mbox{det}(F_m) = \sL_m^2$ (recall that $\deg(F_m) = \deg(\mbox{det}(F_m)) = 0$).

We define $E_m = F_{m+m_0+1}\tensor \sL_{m+m_0+1}^{-1}$, for $m\geq 0$.

Then  
$\mbox{det}(E_m) = \mbox{det}(F_{m+m_0+1})\tensor \sL_{m+m_0+1}^{-2} = \sO_X$. This 
proves Assertion~(1).

\vspace{5pt}

Note that 
$$F^{k*}{E_m} = F^{k*}F_{m+m_0+1}\tensor \sL_{m+m_0+1}^{-p^k} = 
F_{m-k+m_0+1}\tensor \sL_{m+m_0+1}^{-p^k},$$
hence for any $m\geq 0$, 
the bundles $E_m, F^*E_m,\ldots, F^{{m-1}*}E_m$ are semistable.
Since $F^{m*}E_m = F_{m_0+1}\tensor \sL^{-p^m}_{m+m_0+1}$, it has 
 the  HN filtration 
$L_m\subset F^{m*}E_m$ if and only if 
$F^{m_0*}(L_m\tensor\sL_{m+m_0-1}^{p^m})\subset   F^{m_0*}F_{m_0+1} =F_1$ 
is the the HN filtration
of $F_1$.  
Therefore, by the uniqueness of the HN 
filtration we have 
$\deg L_m = (g-1)/p^{m_0}$. 
Moreover $p^{m_0} < g$ as $F_{m_0+1}$ is not semistable.
This proves Assertion~(2).

\vspace{5pt}

Let $\sO_X(1) = K_X^3$, where $K_X$ is the canonical line bundle on $X$.
Since $g\geq 2$, the line bundle $\sO_X(1)$ is very ample  on $X$
(Chap~IV, [H]). Then  (this is a 
standard argument in the literature) we

\vspace{5pt} 

\noindent{\bf Claim}.\quad For $m\geq 1$, the bundle $E_m$ is $3$-regular, 
{\em i.e.}, $H^1(X, E_m(n-1)) = 0$, for $n\geq 3$.

\vspace{5pt}

\noindent{\underline{Proof of the claim}}:\quad By Serre duality 
 $H^1(X, E_m(n-1)) = \Hom(E, \omega_X(1-n))^\vee$.
If $E_m \longto \omega_X(1-n)$ is a nonzero map then the semistability property of 
the sheaf $E_m$ implies 
$\mu(E_m) \leq \mu(\omega_X(1-n))$.
Therefore 
$0\leq (2g-2) + (1-n)\deg \sO_X(1)$. This proves the claim.

Hence (Chapter~14, [Mu2]), for $m\geq 1$, every
 $E_m(3)$ is generated by its global sections. Moreover, we can 
choose $n_0\geq 3$ (Theorem~5.17, [H]) such that 
$E_0(n_0)$ is generated by its global sections. Hence Assertion~(3) follows by taking
$\sL = \sO_X(n_0) = K_X^{3n_0}$.

Moroever each $E_m\tensor\sL$ has the same Hilbert polynomial 
with respect to $\sO_X(1)$ (as each $E_m$ has the same rank and degree). Therefore
the family $\{E_m\tensor\sL\}_{m\geq 0}$ is a bounded family. 
\end{proof}

\begin{rmk}\label{r1}
Lemma~\ref{l1}  implies that, for any prime $p$ and $g>1$, there is a 
 nonsingular curve and a bounded 
family $\F$ of vector bundles on $X$,
such that if  $m_V$ denotes the minimum integer $m$ for which 
$F^{m*}V$ achieves the strong HN filtration then the set 
$\{m_V\mid V\in~~\mbox{the bounded family}~~\F\}$ 
is unbounded.
\end{rmk}

\vspace{5pt}

\noindent{\underline {Proof of Theorem}~\ref{t1}:\quad For given $p$ and $g$, 
we select a nonsingular curve $X$ and a family $\{E_m\}_{m\geq 0}$  of bundles 
and a line bundle $\sL = K_X^{3n_0}$, for some $n_0\geq 3$. 
as in Lemma~\ref{l1}. Since $E_m$ is a vector bundle of rank two over a curve, 
 the (globally generated) bundle $E_m\tensor \sL$ is generated by $3$
global sections (Ex.~8.2, Chap~II, [H]).  
Hence there is a short exact sequence of $\sO_X$-modules
$$ 0 \longto M_m \longto \sO_X\oplus \sO_X \oplus\sO_X\longto 
E_m\tensor \sL \longto 0.$$
Now $M_m = \det(E_m\tensor \sL)^\vee = (\sL^{\tensor 2})^\vee.$
Dualizing the above short exact sequence we get
\begin{equation}\label{e1} 0 \longto (E_m\tensor {\sL})^\vee  \longto \sO_X\oplus \sO_X \oplus\sO_X
\longby{\eta} \sL^{\tensor 2} \longto 0.\end{equation}
Let $$R = \oplus_{n\geq 0}R_n = \oplus_{n\geq 0}H^0(X, \sL^{\tensor 2n})~~
{and}~~I_m = h_{m1}R+h_{m2}R+h_{m3}R,$$ 
where the map $\eta$ is 
 given by 
$(s_1, s_2, s_3) \mapsto h_{m1}s_1+h_{m2}s_2+h_{m3}s_3$. 
Since $\sL^{\tensor 2} = K_X^{\tensor 6n_0}$, for some integer $n_0$, the ring 
$R$ is  
a normal $\Q$-Gorenstein domain.
Let 
${\bf m}$ be the graded maximal ideal of $R$. Note that
$h_{m1}, h_{m2}, h_{m3}\in R_1$ and $\deg~X = e_0(R, {\bf m}) = \deg \sL^{\tensor 2}$. 
By Remark~\ref{r2}, we have
$$c^{I_m}({\bf m}) = 1-a_{min}((E_m\tensor \sL)^\vee)/\deg(\sL^{\tensor 2}).$$
Now the exact sequence 
$0\longto L_m \longto F^{m*}E_m\longto L_m^{-1}\longto 0$
gives
$$0\longto L_m\tensor F^{m*}(\sL^\vee)\longto F^{m*}((E_m\tensor\sL)^\vee)\longto 
L_m^{-1}\tensor F^{m*}(\sL^\vee)\longto 0.$$
and also the strong HN filtration $0\subset L_m\tensor F^{m*}(\sL^\vee)\subset 
F^{m*}(E_m\tensor\sL)^\vee$ of $F^{m*}(E_m\tensor\sL)^\vee$. 

Hence 
$$a_{min}(E_m\tensor\sL)^\vee = 
\frac{\mu_{min}(F^{m*}(E_m\tensor\sL)^\vee)}{p^m} = -\deg(\sL)-\frac{\deg(L_m)}{p^m}
= -\deg(\sL)-\frac{(g-1)}{p^{m+m_0}}.$$
Therefore
$$c^{I_m}({\bf m}) = 1+\frac{1}{2\deg(\sL)}\left[\deg~\sL +   
\frac{g-1}{p^{m+m_0}}\right] = \frac{3}{2}+
\frac{(g-1)}{dp^{m+m_0}},$$
where $d = e_0(R,{\bf m}) = \deg\sL^{\tensor 2}$. This proves the theorem.\hfill\hfill\hspace{50pt} $\Box$

\begin{rmk}
We recall that when $R$ is a regular local ring, then, apart from the 
set of $F$-thresholds (of an ideal)  being discrete and rational,  there can be no strictly 
decreasing sequence of  $F$-thresholds of an ideal $I$ 
(Remark~2.9, [MTW]). This is because in the regular case there is a bijection between 
the set of $F$-thresholds of $I$ and the the set of  test ideals 
of $I$, given by $c\mapsto \tau(I^c)$ such that 
if $c_1$ and $c_2$ are two 
$F$-threhsolds of $I$ then $c_1 < c_2$ if and only of 
$\tau(I^{c_2}) \subset \tau(I^{c_1})$. 

Hence the above example in Theorem~\ref{t1}
shows 
that any ``order reversing''  such bijective correspondence between  the set of 
$F$-thresholds and a  
set of  ideals of some kind, would not hold.
\end{rmk}

\section{$F$-thresholds {\em reduction mod} $p$}

We follow  notations and definitions as given in  Section~2.

\vspace{5pt}

\noindent{\underline {Proof of Theorem}~\ref{t2}~:\quad First we prove 
the following claim.

\vspace{5pt}

\noindent{\bf Claim}.\quad Let $V$ be a vector bundle of rank $r$ on 
a nonsingular curve of genus $g$ over a  field of $\Char~p>0$. If 
$p > \mbox{max}\{4(g-1)r^3, r!\}$  then 
$$a_{min}(V) < \mu_{min}(V) \implies
\mu_{min}(V) = a_{min}(V) + {a}/{pb},$$ where $a$, $b$ are positive 
integers and $\mbox{g.c.d.}(a, p) = 1$ and $0< a/b \leq 4(g-1)(r-1)$.

\vspace{5pt}

\noindent{\underline{Proof of the claim}}:\quad Let $m$ be an integer such 
that $F^{m*}(V)$ achieves
the  strong HN filtration. Note that, by the hypothesis $m\geq 1$ and, by
definition $a_{min}(V) = 
\mu_{min}(F^{m*}V)/p^m$.
By Lemma~1.14 of [T], 
$$\mu_{min}(F^{m*}V)/p^m + C/p = \mu_{min}(V), ~~\mbox{where}~~ 
0 < C \leq 4(g-1)(r-1).$$
This implies $C p^{m-1}(r!) \in \N$ and we can write 
$$\mu_{min}(V) = a_{min}(V) + \frac{C p^{m-1} (r!)}{p^m (r!)} =   
a_{min}(V) + \frac{a}{p b},$$
where $a$ $b$ are positive integers such that $\mbox{g.c.d.}(a, p) = 1$.
This proves the claim.

By Remark~\ref{r2},  if
$c^{I_p}({\bf m}_p) \neq c^I_{\infty}({\bf m})$ then $c^{I_p}({\bf m}_p) 
> c^I_{\infty}({\bf m})$. Moreover  
there exists a vector bundle $W$ on $X = \mbox{Proj}~S$, where 
$\pi:R\longto S$ is the normalization of $R$ and 
 $X$ is a nonsingular projective curve  of degree 
$d$ such that, for $p\gg 0$, 
$$c^I_{\infty}({\bf m}) = 1-\mu_{min}(W)/d\quad \mbox{and}\quad  
c^{I_p}({\bf m}_p) = 1-a_{min}(W_p)/d),$$ 
where  $W_p$ 
denotes the bundle $W$ {\em reduction 
mod} $p$, (similarly for $I$ and ${\bf m}$) and where $\Char~R_p = p $.
The openess property of the semistable bundle ([Ma]) imply that 
 for $p\gg 0$, $\mu_{min}(W_p) = \mu_{min}(W)$. 
Therefore we can write
$$c^{I_p}({\bf m}_p) = 1-\frac{\mu_{min}(W_p)}{d} + \frac{a}{pb}.$$
Since $-\mu_{min}(W_p) = d_1/r_1$, where $d_1, r_1\in \Z_{+}$ such that 
$r_1 <\mu(I)$, the theorem follows for $p\gg 0$. 
\hfill\hfill\hspace{50pt} $\Box$

\begin{rmk}\label{r3} We recall  Example~6.9 from [TrW].
Let $R_p= k[x,y,z]/(h)$, where $h= x^{d-1}y+y^{d-1}z+z^{d-1}x$ and $d\geq 7$ is an 
odd integer. Let $\Char~R_p = p$ where $p \geq d^2$ such that 
$p\equiv\pm 2\pmod{(d^2-3d+3)}$. Then 
$c^{\bf m}({\bf m}) = (3pd+d^2-9d+15)/2pd$.
 Therefore if 
 $c^{{\bf m}_p}({\bf m}_p)\neq c^{\bf m}_{\infty}({\bf m})$ then though $p$ divides 
the denominator of
$c^{{\bf m}_p}({\bf m}_p)$, the denominator need not  be a power of $p$.   
\end{rmk}


\begin{thebibliography}{}

\bibitem[BMS1]{BMS1}{Blickle, M., Musta\c{t}\u{a}, M., Smith, K.},
 {\it F-thresholds of hypersurfaces},
 Trans. Amer. Math. Soc. 361 (2009), no. 12, 6549-6565.
\bibitem[BMS2]{BMS2}{Blickle, M., Musta\c{t}\u{a}, M., Smith, K.},{\it
Discreteness and rationality of F-thresholds},
Michigan Math. J., 57 (2008), pp. 43-61 (Special volume in honor of
Melvin Hochster).


\bibitem[BS]{BS}{Bhargav, B., Singh, A.}, {\it The $F$-pure 
threshold of a Calabi-Yau hypersurface}, Math. Ann. (2015) 362, 551-567.

\bibitem[BSTZ]{BSTZ}{Blickle, M., Schwede, K., Takagi, S., Zhang, W.},
{\it Discreteness and rationality of F-jumping numbers on singular varieties}.
Math. Ann. 347 (2010), no. 4, 917-949.
 
\bibitem[CHSW]{CHSW}{Canton, E., Hernández, D., Schwede, K., Witt,E.}, {\it On 
the behavior of singularities at the F-pure threshold}, 
Illinois J. Math. 60 (2016), no. 3-4, 669-685. 

\bibitem[G]{G}{Gieseker, D.}, {\it Stable vector bundles and the Frobenius 
morphism}, Ann. Sci. École Norm. Sup. (4) 6 (1973), 95-101. 

\bibitem[GrS]{GrS}{Graf, P.,  Schwede, K.}, {\it Discreteness of $F$-jumping 
numbers at isolated non-$\Q$-Gorenstein points}. Proc. Amer. Math. Soc. 146 
(2018), no. 2, 473-487.

\bibitem[HY]{HY}{Hara, N., Yoshida, K.},
{\it A generalization of tight closure and multiplier ideals} Trans. Am. Math. Soc.
355, 3143-3174 (2003).

\bibitem[H]{H}{Hartshorne, R.}, {\it Algebraic geoemetry}, Springer-Verlag NY (1977).

\bibitem[HMNb]{HMNb}{Huneke, C., Montanera, J.A.,  N\'{u}\~{n}ez-Betancourt, L.},
{\it D-modules, Bernstein–Sato polynomials and F-invariants of direct summands},
Advances in Mathematics
Volume 321, 1 December 2017, Pages 298-325.

\bibitem[KSSZ]{KSSZ}{Katzman, M., Schewde, K., Singh, A., Zhang, W.},
{\it  Rings of Frobenius 
operators}. Math. Proc. Cambridge Philos. Soc. 157 (2014), no. 1, 151–167. 

\bibitem[KLZ]{KLZ}{Katzman, M., Lyubeznik, G., Zhang, W.},
{\it  On the discreteness and rationality of F-jumping coefficients},
J. Algebra 322 (2009), no. 9, 3238-3247.

\bibitem[Ma]{Ma}{Maruyama, M.}, {\it Openness of a family of torsion
 free sheaves}, J. Math. Kyoto Univ. 16-3 (1976), 627-637.

\bibitem[Mu1]{Mu1}{Mumford, D.}, {\it Abelian varieties}, Tata Institute 
of Fundamental Research, Studies in Mathematics, No~5, corrected reprint, 
Hindustan Book Agency, New Delhi (2012).

\bibitem[Mu2]{Mu2}{Mumford, D.},{\it Lectures on curves on an Algebraic Surface},
Annals of Math. studies 59, Princeton University Press, Princeton, NJ (1966).

\bibitem[MTW]{MTW}{Musta\c{t}\u{a}, M., Takagi, S., Watanabe, K.I.}, 
{\it F-thresholds and Bernstein-Sato polynomials}, European
congress of mathematics, 341-364, Eur. Math. Soc., Zurich, 2005.

\bibitem[ST]{ST}{Schwede, K., Tucker, T.}, {\it Test ideals of non-principal 
ideals: Computations, Jumping Numbers, Alterations and Division Theorems}, 
J. Math., Pures Appl. (9) 102 (2014), no. 5, 891-929.

\bibitem[TaW]{TaW}{Takagi, S., Watanabe, K.I.}, {\it On F-pure thresholds}, 
J. Algebra 282 (2004), 278-297.


\bibitem[T]{T}{Trivedi, V.}, {\it Hilbert-Kunz multiplicity and
reduction mod $p$}, Nagoya Math. Journal !85 (2007), 123-141.


\bibitem[TrW]{TrW}{Trivedi, V., Watanabe, K.}, {\it Hilbert-Kunz density functions and 
$F$-threshold}, arXiv:1808.04093v1.


\end{thebibliography}
\end{document}